\documentclass[11pt]{amsart}

\usepackage{amsfonts}
\usepackage{amssymb,stmaryrd, latexsym, amsthm, amsmath, verbatim}
\usepackage{indentfirst}

\usepackage{mathtools}
\usepackage{xcolor}
\allowdisplaybreaks

\newtheorem{thm}{{{Theorem}}}[section]
\newtheorem{prop}[thm]{{Proposition}}
\newtheorem{lem}[thm]{{Lemma}}
\newtheorem{cor}[thm]{{Corollary}}

\newtheorem{exa}[thm]{{Example}}
\newtheorem{remark}[thm]{Remark}

\numberwithin{equation}{section}

\newcommand{\vol}{\mathrm{vol}\,}
\newcommand{\Q}{{\mathbb{Q}}}
\newcommand{\R}{{\mathbb{R}}}
\newcommand{\C}{{\mathbb{C}}}
\newcommand{\A}{{\mathbb{A}}}
\newcommand{\Z}{{\mathbb{Z}}}
\newcommand{\N}{{\mathbb{N}}}

\newcommand{\cS}{{\mathcal{S}}}
\newcommand{\cL}{{\mathcal{L}}}

\begin{document}
\thispagestyle{empty}

\title[Central limit theorem for Hecke eigenvalues]{Central limit theorem for Hecke eigenvalues}
\author{Henry H. Kim, Satoshi Wakatsuki and Takuya Yamauchi}
\date{\today}
\thanks{The first author is partially supported by NSERC grant \#482564. 
The second author is partially supported by JSPS KAKENHI Grant Number (B) No.21H00972. 
}
\subjclass[2010]{11F46, 11F70, 22E55, 11R45}
\address{Henry H. Kim \\
Department of mathematics \\
 University of Toronto \\
Toronto, Ontario M5S 2E4, CANADA \\
and Korea Institute for Advanced Study, Seoul, KOREA}
\email{henrykim@math.toronto.edu}

\address{Satoshi Wakatsuki \\
Faculty of Mathematics and Physics, Institute of Science and Engineering\\
Kanazawa University\\
Kakumamachi, Kanazawa, Ishikawa, 920-1192, JAPAN}
\email{wakatsuk@staff.kanazawa-u.ac.jp}

\address{Takuya Yamauchi \\
Mathematical Inst. Tohoku Univ.\\
 6-3,Aoba, Aramaki, Aoba-Ku, Sendai 980-8578, JAPAN}
\email{takuya.yamauchi.c3@tohoku.ac.jp}

\subjclass[2010]{Primary 11N75, Secondary 11F46}

\keywords{Sato-Tate distribution, Hecke eigenvalues, Central limit theorem}

\maketitle

\begin{abstract} 
In this paper, we obtain the central limit theorem of Hecke eigenvalues in very general setting of split simple algebraic groups over $\Q$, using irreducible characters of compact Lie groups.  
\end{abstract}


\section{Introduction}
Let $S_k$ be the space of holomorphic cusp forms of weight $k$ with respect to $SL_2(\Bbb Z)$.  
Let $\mathcal F_k$ be the set of Hecke eigen newforms in $S_k$. 
For $f\in \mathcal F_k$, let 
$\pi=\pi_f$ be the cuspidal automorphic representation of $GL_2(\Bbb A)$ attached to $f$, and $L(s,\pi)=\sum_{n=1}^\infty a_f(n) n^{-s}$ so that $a_f(p)=2\cos\theta_f(p)$ for each prime $p$. Then
Nagoshi \cite{N} obtained the central limit theorem on the distribution of $\{a_f(p)\}$. 
Namely, for any bounded continuous function $\phi: \Bbb R\rightarrow \Bbb R$, if $\frac {\log k}{\log x}\to\infty$ as $x\to\infty$,

\begin{equation*}
\frac 1{|\mathcal F_k|} \sum_{f\in \mathcal F_k} \phi\left(\frac {\sum_{p\leq x} a_{f}(p)}{\sqrt{\pi(x)}}\right)
\longrightarrow \frac 1{\sqrt{2\pi}}\int_{-\infty}^\infty \phi(t)e^{-\frac {t^2}2}\, dt, \quad x\to\infty.
\end{equation*}

In our earlier paper \cite{KWY3}, we proved similar results for holomorphic Siegel modular forms by direct calculations of certain integrals.
In this paper, we extend the previous works to general setting of split simple algebraic groups, using irreducible characters of compact Lie groups. 
The key observation is that irreducible characters are orthonormal. We use the well-known fact that irreducible representations of a complex Lie group are identified with representations of its maximal compact subgroup, and we use them interchangeably. 

Our results encompass previous known results \cite{KWY3,LNW,N}. For example, for $PGL_n$, the degenerate Schur polynomials are exactly irreducible characters of $SL_n(\Bbb C)$, and so our result implies \cite[Theorem 1.2]{LNW}. We demonstrate our result in Section \ref{sec:G_2} for an exotic example of the degree 7 standard $L$-function of the exceptional group of type $G_2$. 

We organize our paper as follows. In Section 2, we review unramified Hecke algebras and Sato-Tate measure. In Section 3, we recall families of cusp forms where the simultaneous vertical Sato-Tate distribution holds. In Section 4, we prove the central limit theorem for the families in Section 3. In Section 5, we prove the central limit theorem for Langlands $L$-functions. In Section 6, we consider the degree 7 standard $L$-function of the exceptional group of type $G_2$. 

\section{Unramified Hecke algebra and Sato-Tate measure}

\subsection{Notation}

Let $G$ denote a connected algebraic group over the rational number field $\Q$.
Throughout this paper, we assume that $G$ is split over $\Q$ and (absolutely) simple. 
For a field $F$ of characteristic zero, let $G(F)$ denote the set of $F$-rational points of $G$. 
For each prime $p$, we write $\Q_p$ for the $p$-adic number field and set $G_p\coloneqq G(\Q_p)$. 
Let $K_p$ denote a hyperspecial compact subgroup of $G_p$. 
We normalize the Haar measure $dg_p$ on $G_p$ by setting $\vol(K_p)=1$ for any prime $p$.
Let $C^\infty_c(G_p)$ denote the space of compactly supported locally constant functions on $G_p$. 
The $L^1$-norm $\| \; \|_{1,p}$ on $C^\infty_c(G_p)$ is defined by $\|h\|_{1,p}\coloneqq \int_{G_p} |h(g_p)|\, dg_p$, $h\in C_c^\infty(G_p)$.  

The rank of $G$ is denote by $\mathbf{r}$. 
Take a maximal $\Q$-split torus $T$ of $G$, i.e., $T(\Q)\cong (\Q^\times)^\mathbf{r}$. 
We also take a Borel subgroup $B=TU$ over $\Q$ where $U$ is the unipotent radical of $B$. 
Let $R$ denote the set of roots of $T$ in $G$, and $R^+$ denote the set of positive roots in $R$ corresponding to $B$. 
Write $W$ for the Weyl group of $T$ in $G$. 
We identify $W$ with the Weyl group of $R$. 
Let $X_*(T)$ denote the abelian group of co-characters of $T$. 
Take a $\Z$-basis $\{e_1,e_2,\ldots,e_\mathbf{r}\}$ of $X_*(T)$. 
For each $\lambda\in X_*(T)$, there uniquely exist $a_1,a_2,\dots,a_\mathbf{r}\in\Z$ so that $\lambda=a_1e_1+a_2e_2+\cdots+a_\mathbf{r}e_\mathbf{r}$, and then we define a height function $\|\; \|$ on $X_*(T)$ by  
\[
| \lambda|\coloneqq \max_{1\le j\le \mathbf{r}} |a_j|, \qquad \|\lambda\|\coloneqq \max_{w\in W} |w\lambda|.
\]

\subsection{Unramified Hecke algebra}\label{sec:Hecke}


In this section, we review the Sato-Tate measure following the formulation in \cite{ST}.
The group $T_p\coloneqq T(\Q_p)$ is a maximal split $\Q_p$-torus of $G_p$, and we have the Cartan decomposition $G_p=K_p T_p K_p$ and the Iwasawa decomposition $G_p=B_pK_p$ where $B_p\coloneqq B(\Q_p)$.


The unramified Hecke algebra $\mathcal{H}^\mathrm{ur}(G_p)\coloneqq C^\infty_c(K_p\backslash G_p/K_p)$ is generated by the characteristic functions $\tau_{\omega,p}$ of the double cosets $K_p\omega(p)K_p$ with ${\omega\in X_*(T)}$. 
For each $\kappa\in\mathbb{N}$, a truncated unramified Hecke algebra $\mathcal{H}^\mathrm{ur}_\kappa(G_p)$ is defined by
\[
\mathcal{H}^\mathrm{ur}_\kappa(G_p)\coloneqq \langle \tau_{\omega,p} \mid  \; \omega\in X_*(T) ,\;   \|\omega\|\leq\kappa \rangle.
\]
\begin{lem}\label{lem:1}
There exists constants $a$, $b\in\N$ such that, for any prime $p$ and any $h\in \mathcal{H}^\mathrm{ur}_\kappa(G_p)$, we have
\[
 \|h\|_{1,p} \ll p^{a\kappa+b} \, \sup_{x\in G_p}|h(x)|. 
\]
\end{lem}
\begin{proof}
We write $|\cdot |_p$ for the valuation of $\Q_p$. 
Without loss of generality, we may assume that $G$ is a closed subgroup of $SL_N$ over $\Q$, $T$ consists of diagonal matrices in $SL_N$, and $U$ consists of strict upper triangular matrices in $SL_N$.  

There exist positive integers $r$ and $s$ such that, for any prime $p$, any $\omega\in X_*(T)$, and any $g=(g_{ij})_{1\leq i,j\leq N}\in G_p\subset SL_N(\Q_p)$, we have $|g_{ij}|_p \le p^{r\| \omega\|+s}$ $(\forall i,j)$ if $g$ is in $K_p \omega(p) K_p$. 
If we suppose $tu\in K_p\omega(p) K_p$, $t=\mathrm{diag}(t_1,t_2,\dots,t_N)\in T_p$, $u\in U_p\coloneqq U(\Q_p)$, then by $t\in SL_N(\Q_p)$ we obtain
\[
p^{-N(r\| \omega\|+s)}\le |t_j|_p \le p^{r\| \omega\|+s} \quad (1\le j\le N).
\]
Hence, by the Iwasawa decomposition, we obtain
\[
K_p \omega(p) K_p  \subset \bigcup_{t\in T_p, u\in U_p } tu K_p 
\]
where $t$ runs over all elements $\mathrm{diag}(p^{m_1},p^{m_2},\dots,p^{m_N})\in T_p$, $-r\| \omega\|-s\le m_j\le N(r\| \omega\|+s)$ and $u$ runs over all elements $(u_{ij})\in U_p$ such that $u_{ij}\in p^{-r\| \omega\|-s}\Z_p/p^{N(r\| \omega\|+s)}\Z_p$ for $i<j$.  
Therefore, since $\|\tau_{\omega,p}\|_{1,p} = \#( K_p \omega(p) K_p /K_p )$, we have $\|\tau_{\omega,p}\|_{1,p} \ll p^{r'\|\omega\|+v'}$ for some $r'$, $s'\in\N$.  
Thus, the proof is completed. 
\end{proof}

\subsection{Plancherel measure}\label{sec:pl}

Write $\widehat{G}$ for the dual group of $G$. 
Let $\widehat{T}$ denote the dual torus of $T$ and let $\widehat{T}_c$ denote the compact subtorus of  $\widehat{T}$. 
The Weyl group $W$ naturally acts on $\widehat{T}$ and $\widehat{T}_c$.  
Take a prime $p$. 
We write $G_p^{\wedge,\mathrm{ur}}$ (resp. $G_p^{\wedge,\mathrm{ur,temp}}$) for the unramified (resp. unramified and tempered) part of the unitary dual of $G_p$. 
By the canonical map given in \cite[pp. 33--34]{ST}, we have the topological injective mapping $G_p^{\wedge,\mathrm{ur}} \to \widehat{T}/W$ and the topological isomorphism
\[
G_p^{\wedge,\mathrm{ur,temp}} \cong \widehat{T}_c/W .
\]
For $h_p\in \mathcal{H}^\mathrm{ur}(G_p)$, a continuous function $\widehat{h_p}$ on $\widehat{T}/W$ is defined by
\[
\widehat{h_p}(\pi_p)\coloneqq \mathrm{tr} \, \pi_p(h_p), \qquad \pi_p \in G_p^{\wedge,\mathrm{ur}},
\]
and it is well-known that the Plancherel measure $\mu_p^\mathrm{pl,ur}$ on $G_p^{\wedge,\mathrm{ur}}$ satisfies
\[
\mu_p^\mathrm{pl,ur}(\widehat{h_p})=h_p(1) , \qquad  h_p\in \mathcal{H}^\mathrm{ur}(G_p). 
\]
Notice that the support of $\mu_p^\mathrm{pl,ur}$ is included in $G_p^{\wedge,\mathrm{ur,temp}} \cong \widehat{T}_c/W$.
Let  $\mu_p$ denote the measure of $\mu_p^\mathrm{pl,ur}$ pulled back to $\widehat{T}_c/W$. 
In \cite[Proposition 3.3]{ST}, $\mu_p$ is explicitly described as
\begin{equation}\label{eq:pl}
\mu_p(t)=C_p\cdot \frac{\det\left((1-\mathrm{ad}(t))|_{\mathrm{Lie}(\widehat{G})/\mathrm{Lie}(\widehat{T})}\right)}{\det\left((1-p^{-1}\mathrm{ad}(t))|_{\mathrm{Lie}(\widehat{G})/\mathrm{Lie}(\widehat{T})}\right)} \, dt
\end{equation}
for some constant $C_p\in\C^\times$, which depends only on the normalization of a Haar measure $dt$ on $\widehat{T}_c/W$. 
If $h_p$ is the characteristic function of $K_p$, then $\widehat{h_p}\equiv 1$ on $\widehat{T}/W$.  
Hence, we obtain 
\begin{equation}\label{eq:constant}
\int_{\widehat{T}_c/W} \mu_p(t)=1.
\end{equation}

Let $S$ be a finite set of prime numbers.
Suppose that $S$ has no intersection with $S_0$, and  set
\[
G_S \coloneqq\prod_{p\in S}G_p,\quad \mathcal{H}^\mathrm{ur}(G_S)\coloneqq\bigotimes_{p\in S}\mathcal{H}^\mathrm{ur}(G_p), \quad \text{and} \quad \mathcal{H}^\mathrm{ur}_\kappa(G_S)\coloneqq\bigotimes_{p\in S}\mathcal{H}^\mathrm{ur}_\kappa(G_p).
\]
Write $G_S^{\wedge,\mathrm{ur}}$ (resp. $G_S^{\wedge,\mathrm{ur,temp}}$) for the unramified (resp. unramified and tempered) part of the unitary dual of $G_S$. 
Clearly, one has an isomorphism 
\[
G_S^{\wedge,\mathrm{ur,temp}} \cong \prod_{p\in S} \widehat{T}_c/W .
\]
Further, for each $h_S\in \mathcal{H}^\mathrm{ur}(G_{S})$ we define the continuous function $\widehat{h_S}$ on $G_S^{\wedge,\mathrm{ur}}$ by $\widehat{h_S}(\pi_S)\coloneqq \mathrm{Tr}\, \pi_S(h_S)$ $(\pi_S\in G_S^{\wedge,\mathrm{ur}})$.
Since the Plancherel measure $\mu_S^\mathrm{pl,ur}$ on $G_S^{\wedge,\mathrm{ur}}$ satisfies $\mu_S^\mathrm{pl,ur}=\prod_{p\in S} \mu_p^\mathrm{pl,ur}$, one has  $\mu_S^\mathrm{pl,ur}(\widehat{h_S})=h_S(1)$ and the support of $\mu_S^\mathrm{pl,ur}$ is contained in $G_S^{\wedge,\mathrm{ur,temp}}$. 
Let  $\mu_S$ denote the measure of $\mu_S^\mathrm{pl,ur}$ pulled back to $\prod_{p\in S} \widehat{T}_c/W $. 

\subsection{Sato-Tate measure}\label{sec:generalST}

We briefly recall a characterization of the Sato-Tate measure for $G$, and refer the reader to \cite[Section 5]{ST} for details. 
We denote by $\mu^\mathrm{ST}_{\infty}$  the Sato-Tate measure  on $\widehat{T}_c/ W$ introduced in \cite[Definition 5.1]{ST}. 
It can be characterized by a limit of Plancherel measures. 
By \cite[Propositions 3.3 and 5.3]{ST} one then has the weak convergence
\begin{equation}\label{eq:plST}
\mu_p\to \mu^\mathrm{ST}_{\infty} \quad \text{as $p\to \infty$.}
\end{equation}
By \eqref{eq:pl}, \eqref{eq:constant}, and \eqref{eq:plST}, we obtain
\begin{equation}\label{eq:ST}
\mu^\mathrm{ST}_\infty(t)=C\cdot \det\left((1-\mathrm{ad}(t))|_{\mathrm{Lie}(\widehat{G})/\mathrm{Lie}(\widehat{T})}\right) \, dt,
\end{equation}
for some constant $C$ so that $\int_{\widehat{T}_c/ W}  \mu^\mathrm{ST}_{\infty}(t)=1$. We also have 
\begin{equation}\label{eq:ptoST}
\mu_p=(1+O(p^{-1}))\mu^\mathrm{ST}_\infty \quad \text{as $p\to \infty$.}
\end{equation}

\subsection{Irreducible characters of compact Lie groups}

Let $\widehat{K}$ denote a maximal compact subgroup of $\widehat{G}$ such that $\widehat{T}_c$ is a maximal torus  in $\widehat{K}$. 
Then, the Weyl group of $\widehat{T}_c$ in $\widehat{K}$ is identified with $W$. 
The root system for $\widehat{K}$ and $\widehat{T}_c$  is irreducible under our setup. 
Let $P$ be the weight lattice for $\widehat{K}$ and $\widehat{T}_c$, and $P^+$ be the set of the dominant weights. 
We regard $e^\lambda$, $\lambda\in P$, as a character of $\widehat{T}_c$ by setting $e^\lambda(x)=e^{\langle \lambda, \log(x)\rangle}$ for $x\in \widehat{T}_c$. 
The characters $e^\lambda$, $\lambda\in P$ generate the group algebra $A=\Bbb RP$. 
The Weyl group $W$ acts on $P$, and also on $A$: $w(e^\lambda)=e^{w\lambda}$. 
Let $\rho=\frac 12 \sum_{\alpha\in R^+} \alpha$, and define 
$$\delta=\prod_{\alpha\in R^+} (e^{\alpha/2}-e^{-\alpha/2}).
$$
Note that $\delta$ is skew-symmetric for $W$, i.e., $w\delta=\epsilon(w)\delta$ for all $w\in W$, where $\epsilon(w)=(-1)^{l(w)}$ and $l(w)$ is the length of the Weyl group element. For each $\lambda\in P^+$, the sum $\sum_{w\in W} \epsilon(w) e^{w(\lambda+\rho)}$ is skew-symmetric, and is divisible by $\delta$. Let
$$\chi_\lambda=\delta^{-1} \sum_{w\in W} \epsilon(w) e^{w(\lambda+\rho)}.
$$
It is called the Weyl character corresponding to $\lambda$. By Weyl's denominator identity, 
$$\delta=\sum_{w\in W} \epsilon(w) e^{w(\rho)}.
$$
Let $\pi_\lambda$ denote the irreducible representation  of $\widehat{K}$ corresponding to the highest weight $\lambda$. 
Notice that $\chi_\lambda$ agrees with the character of $\pi_\lambda$ by the Weyl character formula, and $\chi_\lambda$ is identified with the $W$-invariant polynomial on $\widehat{T}_c$. 
Let $A^W$ denote the subspace of $A$ consisting of $W$-invariant elements in $A$.  
An inner product $\langle \; ,\; \rangle$ on $A^W$ is define by
\[
\langle h_1 ,h_2 \rangle \coloneqq \int_{\widehat{T}_c/W} h_1(t)\, \overline{h_2(t)} \, \mu^\mathrm{ST}_\infty(t) , \quad h_1,h_2\in A^W. 
\]
\begin{lem}\label{lem:2}
The irreducible characters $\chi_\lambda$, $\lambda\in P^+$ form an orthonormal basis on $A^W$. 
In particular, for any $\lambda$, $\lambda'\in P^+$, we have $\langle \chi_\lambda,\chi_{\lambda'}\rangle=\delta_{\lambda,\lambda'}$. 
\end{lem}
\begin{proof}
A function $f$ on $\widehat{K}$ is called a class function if $f(y^{-1}xy)=f(x)$ for all $x,y\in \widehat{K}$. 
By the Peter-Weyl theorem and \cite[Theorem 12.18]{B}, we can prove that the irreducible characters form an orthonormal basis in the $L^2$-space of class functions on $\widehat{K}$.  
Using the Weyl integral formula (see \cite[Corollary 11.32]{B}), we can translate this fact that $\chi_\lambda$, $\lambda\in P^+$ form an orthonormal basis of $A^W$. 
\end{proof}

\begin{remark}
The representation $\pi_\lambda$ of $\widehat{K}$ is self-dual if and only if $\chi_\lambda=\overline{\chi_\lambda}$. 
The adjoint representation is always self-dual. 
It is known that every irreducible representations are self-dual for simply connected compact Lie groups if and only if $-1$ is contained in the Weyl group. It is the case except for $A_n$ ($n\ge 2$), $D_n$ ($n$ is odd), and $E_6$. 
\end{remark}

\begin{cor}\label{chi-lambda-1} 
Suppose $\chi_\lambda\ne 1$.
If $\chi_\lambda=\overline{\chi_\lambda}$, then
\[
\int_{\widehat{T}_c/W} \chi_\lambda(t) \, \mu^\mathrm{ST}_\infty(t)=0,\quad \int_{\widehat{T}_c/W} \chi_\lambda(t)^2 \, \mu^\mathrm{ST}_\infty(t)=1.
\]
\end{cor}
\begin{proof}
This obviously follows from Lemma \ref{lem:2}. 
\end{proof}

\begin{cor}\label{chi-lambda} 
Suppose $\pi_\lambda$ is not self-dual. Then
\[
\int_{\widehat{T}_c/W} \mathrm{Re}(\chi_\lambda(t)) \, \mu^\mathrm{ST}_\infty(t)=\int_{\widehat{T}_c/W} \mathrm{Im}(\chi_\lambda(t)) \, \mu^\mathrm{ST}_\infty(t)=0,
\]

\[\int_{\widehat{T}_c/W}\chi_\lambda(t)^2 \, \mu^\mathrm{ST}_\infty(t)=0.
\]

\[
\int_{\widehat{T}_c/W} \mathrm{Re}(\chi_\lambda(t))^2\, \mu^\mathrm{ST}_\infty(t)=\int_{\widehat{T}_c/W} \mathrm{Im}(\chi_\lambda(t))^2 \, \mu^\mathrm{ST}_\infty(t)=\frac 12.
\]
\end{cor}

\begin{proof} Let $\widehat{\pi_\lambda}=\pi_{\hat\lambda}$ be the dual representation of $\pi_\lambda$. Then $\chi_{\hat\lambda}=\overline{\chi_\lambda}\ne \chi_\lambda$.
Therefore, 
$$\int_{\widehat{T}_c/W}\chi_\lambda(t)^2 \, \mu^\mathrm{ST}_\infty(t)=\int_{\widehat{T}_c/W}\chi_\lambda(t)\overline{\chi_{\hat\lambda}} \, \mu^\mathrm{ST}_\infty(t)=0,
$$
by Lemma \ref{lem:2}. For the last identity, use $\mathrm{Re}(\chi_\lambda(t))=\frac{1}{2}(\chi_\lambda(t)+\overline{\chi_\lambda(t)})$ and $\mathrm{Im}(\chi_\lambda(t))=\frac{1}{2i}(\chi_\lambda(t)-\overline{\chi_\lambda(t)})$.
\end{proof}

\begin{remark} 
We use the well-known fact that irreducible representations of a complex Lie group are identified with representations of its maximal compact subgroup, and we use them interchangeably.
\end{remark}

\begin{remark} The irreducible character $\chi_\lambda$ is an example of Heckman-Opdam polynomials for the root system $R$ (cf. \cite{V}).
\end{remark}

Let $m_\lambda=\sum_{\mu\in W(\lambda)} e^\mu$, where $W(\lambda)$ is the orbit of $\lambda$ under the action of the Weyl group.
Then 
$$m_\lambda=\sum_{\mu\in P^+\atop \mu\preceq \lambda} a_{\lambda,\mu} \, \chi_\mu,
$$
for some $a_{\lambda,\mu}\in \Bbb R$. 




\begin{exa}[$PGSp_4$ case]

We have $\widehat{PGSp_4}=Sp_4(\Bbb C)$. It is of type $C_2$. Hence the Sato-Tate measure is, for $(e^{i\theta_1},e^{i\theta_2}, e^{-i\theta_2},e^{-i\theta_1})\in Sp_4(\Bbb C)$, a torus element, 
$$\mu_\infty^{\rm ST}=\frac 1{4\pi^2}\left| (1-e^{2i\theta_1})(1-e^{2i\theta_2})(1-e^{i(\theta_1-\theta_2)})(1-e^{i(\theta_1+\theta_2)})\right|^2 \, d\theta_1 d\theta_2.
$$

Let $\lambda=e_1$ be a fundamental weight. Then 
$$m_\lambda=e^{i\theta_1}+e^{-i\theta_1}+e^{i\theta_2}+e^{-i\theta_2}.
$$
It is exactly the irreducible character $\chi_\lambda$ associated to $\lambda=e_1$. Let $x=e^{i\theta_1}+e^{-i\theta_1}=2\cos\theta_1$ and $y=e^{i\theta_2}+e^{-i\theta_2}=2\cos\theta_2$. Then $\chi_\lambda=x+y$. 

Let $\lambda=e_1+e_2$. Then $m_\lambda=e^{i(\theta_1+\theta_2)}+e^{-i(\theta_1+\theta_2)}+e^{i(\theta_1-\theta_2)}+e^{-i(\theta_1-\theta_2)}=xy$.
The irreducible character associated to $\lambda=e_1+e_2$ is $\chi_\lambda=xy+1$.

Let $\lambda=2e_1$. Then $m_\lambda=e^{2i\theta_1}+e^{-2i\theta_1}+e^{2i\theta_2}+e^{-2i\theta_2}=x^2+y^2-4$.
The irreducible character associated to $\lambda=2e_1$ is $\chi_\lambda=x^2+y^2-2+xy$.
Let $V(n,m)$ be the highest weight representation associated to $ne_1+me_2$, $n\geq m$. Then $V(1,0)$ is the standard representation, and its weights are $\pm e_1, \pm e_2$, and $V(2,0)=Sym^2\, V(1,0)$, and its weights are $\pm 2e_1, \pm 2e_2, \pm (e_1+e_2), \pm (e_1-e_2), 0$ (with multiplicity 2).

\end{exa}

\begin{exa}[$Sp_{2n}$ case]

We have $\widehat{Sp_{2n}}=SO_{2n+1}(\Bbb C)$. It is of type $B_n$. Hence the Sato-Tate measure is, 
for $(e^{i\theta_1},...,e^{i\theta_n},1,e^{-i\theta_n},...,e^{-i\theta_1})\in SO_{2n+1}(\Bbb C)$, a torus element, 

$$
\mu^{\rm ST}_\infty=\frac 1{(2\pi)^n} \prod_{j=1}^n \left|1-e^{i\theta_j}\right|^2 \prod_{1\leq j<k\leq n} \left|1-e^{i(\theta_j-\theta_k)}\right|^2 \prod_{1\leq j<k\leq n} \left|1-e^{i(\theta_j+\theta_k)}\right|^2 \, d\theta_1\cdots d\theta_n.
$$

Let $\lambda=e_1$ be the fundamental weight. Then 
$$m_\lambda=e^{i\theta_1}+e^{-i\theta_1}+\cdots+e^{i\theta_n}+e^{-i\theta_n}.
$$
Hence the irreducible character associated to $\lambda=e_1$ is
$$\chi_\lambda=1+e^{i\theta_1}+e^{-i\theta_1}+\cdots+e^{i\theta_n}+e^{-i\theta_n}.
$$

Let $x_j=e^{i\theta_j}+e^{-i\theta_j}$. Then $\chi_\lambda=1+x_1+\cdots+x_n$.
\end{exa}

\begin{exa}[$PGL_n$ case]

We have $\widehat{PGL_n}=SL_n(\Bbb C)$. It is of type $A_{n-1}$. Hence the Sato-Tate measure is, for $(e^{i\theta_1},...,e^{i\theta_n})\in SL_n(\Bbb C)$, a torus element, 
$$\mu_\infty^{\rm ST}=\frac 1{n! (2\pi)^{n-1}} \prod_{1\leq j<k\leq n} \left|1-e^{i(\theta_j-\theta_k)}\right|^2 \, d\theta_1\cdots d\theta_{n-1}.
$$
Let $x_j=e^{i\theta_j}$ $(1\le j\le n)$ $(x_1x_2\cdots x_n=1)$. 
Then the irreducible character associated to $\lambda=(1,0,...,0)$, is $\chi_\lambda=x_1+\cdots+x_n$. 

\end{exa}

\subsection{Satake transform}

Fix a prime $p$. 
For each unramified representation $\pi_p \in G_p^{\wedge,\mathrm{ur}}$, we write $s(\pi_p)\in \widehat {T}/W$ for the Satake parameter associated to $\pi_p$. 
Let $\mathcal{S}_p$ denote the Satake transform from $\mathcal{H}^\mathrm{ur}(G_p)$ to $A^W\otimes\C$. (See e.g. \cite{Gr}.) 
In particular, for any $\omega\in X_*(T)$, the image $\mathcal{S}_p(\tau_{\omega,p})$ belongs to $A^W$. 
Then, for $h_p \in \mathcal{H}(G_p)^\mathrm{ur}$, we have
\[
\widehat{h_p}(\pi_p)=\mathcal{S}_p(h_p)(s(\pi_p)) .
\]
Notice that the characters of $\widehat{T}_c$ is uniquely extended on $\widehat{T}$.  
By the definitions for $\mu_p^{\mathrm{pl},\mathrm{ur}}$ and $\mu_p$, we also find
\[
\mu_p^{\mathrm{pl},\mathrm{ur}}(\widehat{h_p})=\mu_p(\mathcal{S}_p(h_p)).
\]
From this point on, we identify $\widehat{h_p}$ with $\mathcal{S}_p(h_p)$.  
Then $\widehat{h_p}$ is regarded as a function on $\widehat{T}/W$ and we have 
\begin{equation}\label{eq:inden}
\mu_p^{\mathrm{pl},\mathrm{ur}}(\widehat{h_p})=\mu_p(\widehat{h_p})=\int_{\widehat{T}_c/W} \widehat{h_p}(t) \, \mu_p(t).
\end{equation}
In addition, we can write it as 
\begin{equation}\label{eq:h_p}
\widehat{h_p}=\sum_{\lambda\in P^+} c_{\lambda} \cdot \chi_{\lambda} ,\quad c_\lambda\in\C \quad \text{(finite sum)}
\end{equation}
so that $\widehat{h_p}(\pi_p)=\sum_{\lambda\in P^+} c_\lambda \cdot \chi_{\lambda}(s(\pi_p))$.
Then
\[
\mu_p^{\mathrm{pl},\mathrm{ur}}(\widehat{h_p})=\mu_p(\widehat{h_p})=\sum_{\lambda\in P^+} c_\lambda \cdot \mu_p(\chi_{\lambda}).
\]
Note that there is only one function $h_p\in\mathcal{H}^\mathrm{ur}(G_p)$ satisfying \eqref{eq:h_p} for a given $\{ c_\lambda \}_{\lambda\in P^+}$,  since the Satake transform $\mathcal{S}_p$ is an isomorphism between $\mathcal{H}^\mathrm{ur}(G_p)$ and $A^W\otimes\C$. 

By an isomorphism $X_*(T)\simeq X^*(\widehat{T})$, any $\lambda\in P^+$ is regarded as an element in  $X_*(T)$. 
The functions $\tau_{\lambda,p}$, $\lambda\in P^+$ form a basis of $\mathcal{H}^\mathrm{ur}(G_p)$. 
\begin{lem}\label{lem:1222}
Take an element $\lambda\in P^+$ and suppose that $\tau_\mu\in \mathcal{H}^\mathrm{ur}_\kappa(G_{p})$ for any $\mu\preceq \lambda$.  
Then, there exist constants $a_1$, $b_1\in\mathbb{N}$ such that, for any prime $p$, we have 
\[
\sup_{x\in G_p} |\mathcal{S}_p^{-1}(\chi_\lambda)(x)|\ll  p^{a_1\kappa+b_1}.  
\]
\end{lem}
\begin{proof}
By \cite[(3.9)]{Gr}, 
\[
\mathcal{S}_p(\tau_{\lambda,p})=p^{\langle \lambda, \rho \rangle} \chi_\lambda + \sum_{\mu\in P^+,\; \mu\prec\lambda} a_\lambda(\mu) \, \chi_\mu, \quad a_\lambda(\mu)\in\C.
\]
By using $\tau_\lambda\in \mathcal{H}^\mathrm{ur}_\kappa(G_{p_j})$, the inner product $\langle \; , \; \rangle$ on $A^W$, and the Cauchy-Schwarz inequality,  one can prove that there exist constants $a_2$, $b_2\in\mathbb{N}$ so that, for any $\mu\prec\lambda$, 
\[
|a_\lambda(\mu)|  = | \langle \mathcal{S}_p(\tau_{\lambda,p}) , \chi_\mu\rangle| \le \langle  \mathcal{S}_p(\tau_{\lambda,p}) , \mathcal{S}_p(\tau_{\lambda,p}) \rangle^{1/2}\le \sup_{t\in\widehat{T}_c} |\mathcal{S}_p(\tau_{\lambda,p})(t)| \ll p^{a_2\kappa+b_2}. 
\]
Note that $a_2$ and $b_2$ are independent of $p$. 
Hence, 
\[
\sup_{x\in G_p} |\mathcal{S}_p^{-1}(\chi_\lambda)(x)|\ll \sup_{x\in G_p}\tau_{\lambda,p}(x)+ p^{a_2\kappa+b_2}\sum_{\mu\in P^+,\; \mu\prec\lambda}  \sup_{x\in G_p} |\mathcal{S}_p^{-1}(\chi_\mu)(x)|.
\]
By an inductive argument for $\lambda$, we obtain the assertion. 
\end{proof}


\section{Asymptotic behaviors of Hecke eigenvalues; Simultaneous vertical Sato-Tate distribution}

\subsection{Notations}
\begin{itemize}
\item Let $\A$ denote the adele ring of $\Q$, and $\A_f$ the finite adele ring of $\Q$. 
\item Set $G_\infty\coloneqq G(\R)$, and take a maximal compact subgroup $K_\infty$ of $G_\infty$. 
Choose a Haar measure $dg_\infty$ on $G_\infty$. 
\item We put $\mathbf{d}\coloneqq \dim G_\infty/K_\infty$. Recall $\mathbf{r}=\mathrm{rank}(G)$. 
\item For a topological group $\mathfrak{G}$, we write $\mathfrak{G}^\wedge$ for the unitary dual of $\mathfrak{G}$.
\item Take a finite set $\mathcal{T}$ of primes  and an open compact subgroup $K_\mathcal{T}$ of $G(\Q_\mathcal{T})$. An open compact subgroup $K$ of $G(\A_f)$ is defined by
\[
K\coloneqq K_\mathcal{T}\times\prod_{p\notin \mathcal{T}}K_p.
\] 
\item Set $\mathcal{H}^\mathcal{T}=\otimes_{p\notin \mathcal{T}}' \mathcal{H}^\mathrm{ur}(G_p)$ and $\| \; \|_1^\mathcal{T} \coloneqq \prod_{p\notin \mathcal{T}} \| \; \|_{1,p}$. 
Let $S$ be a finite set of primes and suppose $S\cap \mathcal{T}=\emptyset$.
For each $h \in \mathcal{H}^\mathrm{ur}(G_S)$, we identify $h$ with $h \otimes (\otimes_{p\notin \mathcal{T}\sqcup S}\mathbf{1}_{K_p})$, where $\mathbf{1}_{K_p}$ is the characteristic function of $K_p$. 
Then, $\mathcal{H}^\mathrm{ur}(G_S)$ is a subalgebra in $\mathcal{H}^\mathcal{T}$. 
We also put
\[
p_S\coloneqq \prod_{p\in S} p.
\]

\item Let $Z$ denote the center of $G$ and we set
\[
v(K)\coloneqq |Z(\Q)\cap G_\infty K| \, \frac{\mathrm{vol}(G(\Q)\backslash  G(\A))}{\mathrm{vol}(K)}.
\]
\end{itemize}

\subsection{Space of cusp forms}
We write $L^2_\mathrm{cus}(G(\Q)\backslash G(\A))$ for the cuspidal spectrum of $L^2(G(\Q)\backslash G(\A))$, and we have its decomposition as
\[
L^2_\mathrm{cus}(G(\Q)\backslash G(\A)) \cong \bigoplus_{\pi\in\widehat{G(\A)}} m_\pi \cdot \pi
\]
where $m_\pi\in\Z_{\ge 0}$ is the multiplicity of $\pi$. 
Take an irreducible representation $\sigma$ of $\widehat{G_\infty}$ and a $K_\infty$-type $\tau$ in $\widehat{K_\infty}$.
For each cuspidal automorphic representation $\pi=\pi_\infty\otimes \otimes_p \pi_p$, we set $\pi_{\bold f}=\otimes_p' \pi_p$, 
$$\pi_{\bold f}^{K}=\{ w\in \pi_{\bold f} \mid \pi_{\bold f}(k)w=w \, \text{ $(\forall k\in K)$} \},
$$  
\[
\cL(\sigma,K)=\{ \phi_\infty\otimes\phi_{\bold f}\in m_\pi \cdot \pi_\infty\otimes \pi_{\bold f}^{K} \mid \text{$\pi_\infty\simeq  \sigma$ and $\phi_\infty$ is $K_\infty$-finite} \} \subset L^2_\mathrm{cus}(G(\Q)\backslash G(\A)).
\]
Then, a vector space $\cS(\sigma,\tau,K)$ is defined by
\[
\cS(\sigma,\tau,K)= (\cL(\sigma,K)\otimes V_{\tau}^\vee)^{K_\infty},
\] 
where $V_{\tau}^\vee$ is the representation space of the contragredient representation of $\tau$. 
Every element  in $\cS(\sigma,\tau,K)$ is a $V_{\tau}^\vee$-valued automorphic form on $G(\Q)\backslash G(\A)$. 
By Harish-Chandra's finiteness theorem, $\dim \cS(\sigma,\tau,K)$ is finite. 
For each $h\in \mathcal{H}^\mathcal{T}$, a Hecke operator $T_h$ on $\cS(\sigma,\tau,K)$ is defined as
\[
(T_h\, \phi) (x)=\int_{G_S} f(xg) \, h(g) \, dg, \qquad f\in \cS(\sigma,\tau,K),
\]
where $dg=\prod_{p\notin \mathcal{T}}dg_p$. 
If $f \in \cS(\sigma,\tau,K)$ is an eigen vector for any $T_h$, then $f$ is said to be a Hecke eigenform. 
By the definition, there exists a basis of $\cS(\sigma,\tau,K)$, which consists of Hecke eigenforms.  
Let $\mathcal{F}(\sigma,\tau,K)$ denote a basis of Hecke eigenforms in $\cS(\sigma,\tau,K)$. 
For each $f\in \mathcal{F}(\sigma,\tau,K)$ and $h\in\mathcal{H}^\mathcal{T}$, we write $a_f(h)$ for the Hecke eigenvalue of $f$ and $T_h$, that is,
\[
T_h \, f=a_f(h) \,f ,\quad a_f(h)\in\C.
\]

\subsection{A Family of Maass cusp forms}\label{sec:Maass}

Let $T_\infty^0$ denote the connected component of $1$ in $T_\infty\coloneqq T(\R)$ (hence, $T_\infty^0\cong (\R_{>0})^{\mathrm{rank}(G)}$), let $\mathfrak{t}$ denote the Lie algebra of $T_\infty^0$. 
Then, we have the Langlands decomposition $G_\infty=T_\infty^0 U_\infty K_\infty$, where $U_\infty\coloneqq U(\R)$. 
Let $\mathfrak{t}^*$ denote the dual vector space of $\mathfrak{t}$, and let $G_\infty^{\wedge,\mathrm{ur},\mathrm{temp}}$ denote the subset of tempered spherical elements in $G_\infty^\wedge$. 
The element $\beta\in i\mathfrak{t}^*$ is identified with the character $t\mapsto e^{\beta ( \log(t))}$ on $T_\infty$.  
Then, an isomorphism $G_\infty^{\wedge,\mathrm{ur},\mathrm{temp}}\cong i\mathfrak{t}^*/W$ is obtained from the induced representation $\sigma_\beta \coloneqq \mathrm{Ind}^{G_\infty}_{T_\infty^0U_\infty}(\beta\boxtimes 1)$. 

For $h_\infty\in C_c^\infty(K_\infty\backslash G_\infty/ K_\infty)$ and $\pi\in G_\infty^{\wedge,\mathrm{ur},\mathrm{temp}}$, we set $\widehat{h_\infty}(\pi)\coloneqq \mathrm{tr} \, \pi(h_\infty)$. 
Define the Plancherel measure $\mu_\infty$ on $G_\infty^{\wedge,\mathrm{ur},\mathrm{temp}}$ by
\[
\mu_\infty (\widehat{h_\infty})=h_\infty(1) ,\qquad h_\infty\in C_c^\infty(K_\infty\backslash G_\infty/ K_\infty).
\]
Take a bounded domain $D$ with a rectifiable boundary in $i\mathfrak{t}^*$.  
There exists a constant $c_D>0$ such that 
\[
\mu_\infty(tD)=c_D\, t^\mathbf{d}+O(t^{\mathbf{d}-1}), \quad t\longrightarrow\infty.
\]
A family $\mathcal{F}(t,K)$ of spherical Maass cusp forms in $L^2$ is defined by
\[
\mathcal{F}(tD,K)\coloneqq \bigcup_{\sigma\in tD} \mathcal{F}(\sigma,\mathbf{1}_\infty,K) \quad \text{(finite sum)} 
\]
where  $\mathbf{1}_\infty$ denote the trivial representation of $K_\infty$. 
Set $\mathfrak{ms}(h)\coloneqq \sup_{x\in G(\A^\mathcal{T}),h(x)\neq 0} \log \| x \|$ for $h\in\mathcal{H}^\mathcal{T}$, where $\A^\mathcal{T}=\otimes_{p\notin \mathcal{T}}' \Q_p$ and $\| \; \|$ is a height on $G(\A)$ as in \cite{Art}. 

\begin{thm}{\rm \cite[Theorem 6.14]{FL}}. \label{thm:FL}
Assume that $G$ is a Chevalley group. 
Then, there exists a constant $\delta>0$ such that, for any $h\in\mathcal{H}^\mathcal{T}$, 
\[
\sum_{f\in \mathcal{F}(tD,K)} a_f(h) =v(K)\, \mu_\infty(tD) \, h(1)+O_{D,K}\left( \|h\|_1\, (1+\mathfrak{ms}(h))^\mathbf{r}\, t^{\mathbf{d}-\delta} \right) ,\quad t\to\infty.
\]
\end{thm}
\begin{thm}{\rm \cite[Theorem 3.11]{FL2} and \cite[Theorem 1.1]{FM}}. \label{thm:FM}
Assume that $G$ is a classical group or an exceptional group of type $G_2$. 
When $G$ is not of types $A_1$ and $A_2$, for any $h\in\mathcal{H}^\mathcal{T}$, we have
\[
\sum_{f\in \mathcal{F}(tD,K)} a_f(h) =v(K)\, \mu_\infty(tD) \, h(1)+O_{D,K}\left( \|h\|_1\, t^{\mathbf{d}-1} \right),\quad t\to\infty.
\]
When $G$ is of type $A_1$ (resp. $A_2$), the remainder term in the above equation is replaced by $\|h\|_1\, t^{\mathbf{d}-\frac{1}{2}} \, \log(1+t)$ (resp. $\|h\|_1\, t^{\mathbf{d}-1} \, (\log(1+t))^2$).
\end{thm}

Note that the factor $v(K) \, \mu_\infty$ does not depend on the choice of a Haar measure $dg_\infty$ on $G_\infty$. 
There are some other general results for Maass forms.
See \cite{E} for quasi-split groups, and \cite{RW} for non-trivial $K_\infty$-types $\tau$ and compact arithmetic quotients. 

\begin{prop}\label{prop:Maass}
Assume that $G$ is a Chevalley  group or a classical group. 
Then there exist $a$, $b\in\N$ and $\delta>0$ such that, for any $\kappa\in\N$, any $S$ such that $S\cap \mathcal{T}=\emptyset$, any $h\in \mathcal{H}^\mathrm{ur}_\kappa(G_S)$,
\[
\frac{1}{v(K)\, \mu_\infty(tD) }\sum_{f\in \mathcal{F}(tD,K)} a_f(h) =  \mu_S^{\mathrm{pl},\mathrm{ur}}( \widehat{h} )+O\left( t^{-\delta} \, p_S^{a\kappa+b}\, \sup_{x\in G_S}|h(x)|  \right),\quad t\to\infty.
\]
\end{prop}
\begin{proof}
This follows from Lemma \ref{lem:1} and Theorems \ref{thm:FL} and \ref{thm:FM}.
\end{proof}

By this proposition, we obtain the simultaneous vertical Sato-Tate theorem as follows.
\begin{prop}\label{prop:Maass2}
Assume that $G$ is a Chevalley  group or a classical group, and let $\delta>0$ be the constant as in Proposition \ref{prop:Maass}. 
Take a sequence $\{ c_\lambda \}_{\lambda\in P^+}$ such that $c_\lambda=0$ except finitely many $\lambda$.  
For each prime $p$, we define a function $h_p\in\mathcal{H}^\mathrm{ur}(G_p)$ by $\widehat{h_p}=\sum_\lambda c_\lambda\cdot \chi_\lambda$, see  \eqref{eq:h_p}. 
Then, there exist positive integers  $a'$, $b'$, $\kappa\in\N$ such that, for any $u\in \mathbb{N}$, any $(r_1,r_2,\ldots,r_u)\in\mathbb{N}^u$, and any distinct primes $p_1$, $p_2,\ldots,p_u$ such that $p_i\notin \mathcal{T}$, 
 we have $h_{p_j}\in \mathcal{H}^\mathrm{ur}_\kappa(G_p)$ and 
\begin{multline*}
\frac{1}{v(K)\, \mu_\infty(tD) }\sum_{f\in \mathcal{F}(tD,K)} a_f(h_{p_1})^{r_1}\, a_f(h_{p_2})^{r_2}\cdots  a_f(h_{p_u})^{r_u} \\
= \int_{(\widehat{T}_c/ W)^u} \widehat{h_{p_1}}(x_1)^{r_1}\cdots \widehat{h_{p_u}}(x_u)^{r_u}\, \mu_S(x_1,\dots,x_u) + O(p_S^{(a' \kappa+ b')r} t^{-\delta}),\quad t\to\infty,
\end{multline*}
where we set $r\coloneqq r_1+\cdots+r_u$ and $S\coloneqq \{ p_1,\ldots,p_u\}$. 
See \eqref{eq:inden} for the meaning of $\widehat{h_{p_j}}(x_j)$. 
\end{prop}
\begin{proof}
For each $\lambda\in P^+$, we set $N(\lambda)\coloneqq \{ \mu \in P^+ \mid \mu\le \lambda\}$. 
Since $h_{p_j}=\sum_\lambda c_\lambda\cdot \mathcal{S}_{p_j}^{-1}(\chi_\lambda)$ is a finite sum, there exists a constant $\kappa$ so that $h_{p_j}\in \mathcal{H}^\mathrm{ur}_\kappa(G_{p_j})$ for any $j$ and $\tau_\mu\in \mathcal{H}^\mathrm{ur}_\kappa(G_{p_j})$ for any $\mu\in\cup_{\lambda\in P^+, c_\lambda\neq 0}N(\lambda)$. 
Under this condition, by Lemma \ref{lem:1222}, 
\begin{equation}\label{eq:1222h}
\sup_{x\in G_{p_j}} |h(x)| \ll p_j^{a_3\kappa+b_3} \qquad  (1\le j\le u) 
\end{equation}
for some $a_3$, $b_3\in \mathbb{N}$. 

 
Let $w_1\in\mathcal{H}^\mathrm{ur}_{\kappa_1}(G_p)$ and $w_2\in\mathcal{H}^\mathrm{ur}_{\kappa_2}(G_p)$, 
and write $w_1*w_2$ for the convolution of $w_1$ and $w_2$. 
The following equations are known or can be easily shown by the known results, cf. \cite{Gr}. 
\begin{align}
& w_1*w_2\in  \mathcal{H}^\mathrm{ur}_{\kappa_1+\kappa_2}(G_p) , \label{eq:ww1}  \\
& \sup_{x\in G_p}|w_1*w_2(x)|\le  \|w_2\|_{1,p}\, \sup_{x\in G_p}|w_1(x)| , \label{eq:ww2} \\
& \widehat{w_1*w_2}=\widehat{w_1}\, \widehat{w_2}. \label{eq:ww3}
\end{align}

We set
\[
h_{r_1,r_2,\ldots,r_u}\coloneqq \bigotimes_{1\le j\le u}h_{p_j}^{(r_j)},\quad  
h_{p_j}^{(r_j)}\coloneqq \overbrace{h_{p_j}*h_{p_j}*\cdots*h_{p_j}}^{r_j}. 
\]
Using \eqref{eq:1222h}, \eqref{eq:ww2}, and Lemma \ref{lem:1}, we have
\[
\sup_{x\in G_S}|h_{r_1,r_2,\ldots,r_u}(x)| \le \prod_{j=1}^u \| h_{p_j} \|^{r_j-1}_{1,p}\sup_{x\in G_{p_j}}|h_{p_j}(x)| \ll p_S^{(a_4\kappa+b_4)r}.
\]
for some $a_4$, $b_4\in\mathbb{N}$. 
By \eqref{eq:ww1}, $h_{r_1,r_2,\ldots,r_u}\in \mathcal{H}^\mathrm{ur}_{r\kappa}(G_S)$. 
In addition, by \eqref{eq:ww3} 
\[
a_f(h_{r_1,r_2,\ldots,r_u})=\widehat{h_{r_1,r_2,\ldots,r_u}}(\pi_{p_1}\otimes \pi_{p_2}\otimes\cdots\otimes \pi_{p_u})=\prod_{j=1}^u \widehat{h_{p_j}}(\pi_{p_j})^{r_j}=\prod_{j=1}^u a_f(\pi_{p_j})^{r_j}.
\]
By \eqref{eq:ww3} and the identification \eqref{eq:h_p}, 
\[
\mu_S^{\mathrm{pl},\mathrm{ur}}(\widehat{h_{r_1,r_2,\ldots,r_u}})= \mu_S(\widehat{h_{p_1}}^{r_1}\widehat{h_{p_2}}^{r_2}\cdots \widehat{h_{p_u}}^{r_u}).
\]
Therefore, the proof is completed by using Proposition \ref{prop:Maass} for $h=h_{r_1,r_2,\ldots,r_u}$.   
\end{proof}

\subsection{A Family of holomorphic and quaternionic cusp forms; weight aspect}\label{hol}

Consider one of the following two cases:
\begin{itemize}
\item Suppose $G=Sp_{2n}$ and $\tau_k$ denotes the $1$-dimensional representation $x\mapsto \det^k(x)$ of $K_\infty\cong U_n$. 
Let $\sigma_k$ denote the holomorphic discrete series of $G_\infty$ with the minimal $K_\infty$-type $\tau_k$. 
\item Suppose $G$ is an exceptional group of type $G_2$ or $F_4$. 
In case of $G_2$ (resp. $F_4$), we have $K_\infty \cong (SU_2\times SU_2)/\{\pm 1\}$ (resp. $K_\infty \cong (SU_2\times Sp_3)/\{\pm 1\}$), where $Sp_3$ is the compact symplectic group of rank 3. 
Write $\tau_k$ for the irreducible representation $\mathrm{Sym}^{2k}\boxtimes \mathbf{1}$ of $K_\infty$, where $\mathrm{Sym}^{2k}$ is the symmetric representation of $SU_2$ and $\mathbf{1}$ is the trivial representation. 
Let $\sigma_k$ denote the quaternionic discrete series of $G_\infty$ with the minimal $K_\infty$-type $\tau_k$. 
\end{itemize}
A congruence subgroup $\Gamma$ of $G(\Q)$ is obtained by $G(\Q)\cap K$.
Under this condition, $\mathcal{F}(\sigma_k,\tau_k,K)$ consists of scalar valued Siegel cusp forms or quaternionic cusp forms of weight $k$ for $\Gamma$. 
Let $d(\sigma_k)$ denote the formal degree of $\sigma_k$. 
By Harish-Chandra's formula for the formal degrees, there exists a constant $c_\infty>0$ such that
\[
d(\sigma_k)=c_\infty \, k^{\frac{\mathbf{d}}{2}+w_G} +O \left(k^{\frac{\mathbf{d}}{2}+w_G-1}\right), \quad w_G\coloneqq \begin{cases} 0 & \text{if $G=Sp_{2n}$}, \\ 1 & \text{if $G$ is of type $G_2$ or $F_4$}. 
\end{cases}
\]
Notice that $\frac{\mathbf{d}}{2}+w_G=\frac{n^2+n}{2}$ when $G=Sp_{2n}$; $\frac{\mathbf{d}}{2}+w_G=5$ when $G$ is of type $G_2$; and  $\mathbf{d}=15$ when $G$ is of type $F_4$. 

\begin{thm}\label{thm:STW}
Then for any $\varepsilon>0$ and any $h\in\mathcal{H}^\mathcal{T}$, 
\[
\sum_{f\in \mathcal{F}(\sigma_k,\tau_k,K)} a_f(h) =v(K)\, d(\sigma_k) \, h(1)+O_{\varepsilon,K}\left( \|h\|_1\, k^{\frac{\mathbf{d}-\mathbf{r}}{2}+w_G+\varepsilon} \right) ,\quad k\to\infty.
\]
\end{thm}

\begin{proof}
This theorem will be proved in the paper in preparation \cite{STW} based on \cite[Theorem 4.1]{FM} of Finis and Matz.
\end{proof}

Note that the factor $v(K) \, d(\sigma_k)$ does not depend on the choice of a Haar measure $dg_\infty$ on $G_\infty$, see \eqref{eq:leadingsp} and \eqref{eq:leadingg2}. 
This formula also implies the asymptotics of dimensions as
\[
\dim \mathcal{S}(\sigma_k,\tau_k,K)=\#\mathcal{F}(\sigma_k,\tau_k,K)= v(K)\, d(\sigma_k) + O_{\varepsilon,K}\left( k^{\frac{\mathbf{d}-\mathbf{r}}{2}+w_G+\varepsilon} \right) ,\quad k\to\infty,
\]
when $h$ is the characteristic function of $K$. 

See the paper \cite{KWY} for a precise asymptotic formula of Siegel cusp forms of degree two $(n=2)$. 
See also the papers \cite{Dalal,ST} for the general results for families of cusp forms related to discrete series. 
Notice that the results in \cite{Dalal,ST} do not cover Theorem \ref{thm:STW}, because their results are available only for the aspect $\min_{\alpha\in R^+} \langle \lambda+\rho, \alpha\rangle\to \infty$, where $R^+$ is a positive root system of $G(\C)$, $\rho=\frac{1}{2}\sum_{\alpha\in R^+}\alpha$, and $\lambda$ varies on highest weights of finite dimensional irreducible algebraic representations of $G_\infty$ corresponding to discrete series. 

\begin{prop}\label{prop:weight}
Then, there exist $a$, $b\in\N$ and $\delta>0$ such that, for any $\kappa\in\N$, any $S$ such that $S\cap \mathcal{T}=\emptyset$, any $h\in \mathcal{H}^\mathrm{ur}_\kappa(G_S)$,
\[
\frac{1}{v(K)\, d(\sigma_k) }\sum_{f\in \mathcal{F}(\sigma_k,\tau_k,K)} a_f(h) =  \mu_S^{\mathrm{pl},\mathrm{ur}}( \widehat{h} )+O\left( k^{-\delta} \, p_S^{a\kappa+b}\, \sup_{x\in G_S}|h(x)|  \right),\quad k\to\infty.
\]
\end{prop}
\begin{proof}
This  follows from Lemma \ref{lem:1} and Theorem \ref{thm:STW}.
\end{proof}

\begin{prop}\label{prop:weight2}
Let $\delta>0$ be the constant as in Proposition \ref{prop:weight}, and let $\{ c_\lambda \}_{\lambda\in P^+}$ and $h_p$ be the same as in Proposition \ref{prop:Maass2}.  
Then, there exist positive integers  $a'$, $b'$, $\kappa\in\N$ such that, for any $u\in \mathbb{N}$, any $(r_1,r_2,\ldots,r_u)\in\mathbb{N}^u$, and any distinct primes $p_1$, $p_2,\ldots,p_u$ such that $p_i\notin \mathcal{T}$, 
 we have $h_{p_j}\in \mathcal{H}^\mathrm{ur}_\kappa(G_p)$ and 
\begin{multline*}
\frac{1}{v(K)\, d(\sigma_k) }\sum_{f\in \mathcal{F}(\sigma_k,\tau_k,K)} a_f(h_{p_1})^{r_1}\, a_f(h_{p_2})^{r_2}\cdots  a_f(h_{p_u})^{r_u} \\
= \int_{(\widehat{T}_c/ W)^u} \widehat{h_{p_1}}(x_1)^{r_1}\cdots \widehat{h_{p_u}}(x_u)^{r_u}\, \mu_S(x_1,\dots,x_u) + O(p_S^{(a' \kappa+ b')r} k^{-\delta}),\quad k\to\infty,
\end{multline*}
where $r\coloneqq r_1+\cdots+r_u$ and $S\coloneqq \{ p_1,...,p_u\}$. 
\end{prop}
\begin{proof}
This can be proved by Proposition \ref{prop:weight} and the same argument as in the proof of Proposition \ref{prop:Maass2} 
\end{proof}

\subsection{Level aspect for Siegel cusp forms}\label{level}
In the case of the Siegel cusp forms, we have obtained asymptotic formulas not only for the weight aspect as described above but also for the level aspect. 
Note that, unlike the previous section, $k$ is fixed, and $K$ and $\mathcal{T}$ vary in the asymptotic formula.
Let  $G=Sp_{2n}$. 
Take a natural number $N$. 
For each prime $p$, we set
\[
K_p(N)\coloneqq \{  g\in K_p \mid g\equiv I_{2n} \mod N\Z_p  \} .
\]
An open compact subgroup $K(N)$ of $G(\A_f)$ is defined to be $K(N)=\prod_p K_p(N)$. Then $v(K(N))=c\, N^{n(2n+1)}\prod_{p\mid N}\prod_{j=1}^n (1-p^{-2j})$ for a constant $c>0$.
Set $\mathcal{T}_N\coloneqq \{ p\mid p$ divides $N\}$. 

\begin{thm}{\rm \cite[Theorem 3.10]{KWY2}}. \label{thm:levelsp}
Suppose that $k>n+1$. 
Then, there exist $a$, $b\in\N$ and $c_0>0$ such that, for any $\kappa\in\N$, any $S$ with $S\cap \mathcal{T}_N=\emptyset$, any $h\in \mathcal{H}^\mathrm{ur}_\kappa(G_S)$, if $N\ge c_0 p_S^{2n\kappa}$,
\begin{multline*}
\sum_{f\in \mathcal{F}(\sigma_k,\tau_k,K(N))} a_f(h) =v(K(N))\, d(\sigma_k) \,  \mu_S^{\mathrm{pl},\mathrm{ur}}( \widehat{h} )+O_k\left( p_S^{a\kappa+b}\, N^{-n} v(K(N)) \, \sup_{x\in G_S}|h(x)| \right) , \\ N\to\infty.
\end{multline*}
\end{thm}
The factor $v(K(N)) \, d(\sigma_k)$ $(N>2)$ of the leading term is explicitly given as
\begin{equation}\label{eq:leadingsp}
v(K(N)) \, d(\sigma_k)= N^{n(2n+1)}\prod_{p\mid N}\prod_{j=1}^n (1-p^{-2j})\times \prod_{t=1}^n \prod_{u=t}^n(2k-t-u)\times \prod_{j=1}^n \frac{(j-1)!\, |B_{2j}|}{(2j-1)!\, j}\times \frac{1}{2^{3n}}
\end{equation}
(see \cite[Theorem 1.6]{W}), where  $B_j$ denotes the $j$-th Bernoulli number. 

The following is the simultaneous vertical Sato-Tate theorem.
\begin{prop}\label{prop:levelsp2}
Suppose that $k>n+1$. 
Let $\{ c_\lambda \}_{\lambda\in P^+}$ and $h_p$ be the same as in Proposition \ref{prop:Maass2}.  
Then, there exist positive integers  $a'$, $b'$, $\kappa\in\N$ such that, for any $u\in \mathbb{N}$, any $(r_1,r_2,\ldots,r_u)\in\mathbb{N}^u$, and any distinct primes $p_1$, $p_2,\ldots,p_u$ such that $p_i\notin \mathcal{T}_N$, 
 we have $h_{p_j}\in \mathcal{H}^\mathrm{ur}_\kappa(G_p)$ and, if $N\ge c_0 p_S^{2n\kappa}$ for $S\coloneqq \{ p_1,\ldots,p_u\}$,
\begin{multline*}
\frac{1}{v(K(N))\, d(\sigma_k) }\sum_{f\in \mathcal{F}(\sigma_k,\tau_k,K(N))} a_f(h_{p_1})^{r_1}\, a_f(h_{p_2})^{r_2}\cdots  a_f(h_{p_u})^{r_u} \\
= \int_{(\widehat{T}_c/ W)^u} \widehat{h_{p_1}}(x_1)^{r_1}\cdots \widehat{h_{p_u}}(x_u)^{r_u}\, \mu_S(x_1,\dots,x_u) + O(p_S^{(a' \kappa+ b')r} N^{-n}),\quad N\to\infty,
\end{multline*}
where $r\coloneqq r_1+\cdots+r_u$.
\end{prop}
\begin{proof}
This can be proved by Theorem \ref{thm:levelsp} and the same argument as in the proof of Proposition \ref{prop:Maass2} 
\end{proof}



\section{Central limit theorem}\label{sec:gencls}

Choose a sequence $\{ c_\lambda \}_{\lambda\in P^+}$ so that
\begin{itemize}
\item $c_\lambda\in\R$.
\item $c_\lambda=0$ except finitely many $\lambda$. 
\end{itemize}
 For all $p$,  a function $h_p\in \mathcal H^{\rm ur}(G_p)$ is defined by  $\widehat{h_p}=\sum_\lambda c_\lambda\cdot \chi_\lambda$; see \eqref{eq:h_p}. 

Take a Hecke eigenform  $f\in \mathcal{F}(\sigma,\tau,K)$. 
Let $\pi_f=\pi_\infty\otimes \otimes_p' \pi_p$ be the cuspidal representation associated to $f$. Let $\mathcal T$ be the set of primes such that if $p\notin \mathcal T$, $\pi_p$ is unramified. Let $S$ be a finite set of primes such that $S\cap \mathcal T=\empty$.
If $h=\otimes_{p\in S} h_p\in \mathcal H^{\rm ur}(G_S)$, then the Hecke eigenvalue $a_f(h)$ satisfies
\[
a_f(h)=\prod_{p\in S} \widehat{h_p}(\pi_p)=\prod_{p\in S}\left( \sum_{\lambda\in P^+} c_\lambda\cdot \chi_\lambda (s(\pi_p))\right).
\]
Assume the following conditions for $h_p$:
\begin{itemize}
\item If $\lambda=0$, then $c_\lambda=0$. 
\item $\pi_\lambda$ is self-dual if $c_\lambda\neq 0$. 
\item $\sum_{\lambda\in P^+}c_\lambda^2=1$. 
\end{itemize}
Then, we have $a_f(h)\in\R$, and
by Corollary \ref{chi-lambda-1},
\begin{equation*} 
\int_{\widehat{T}_c/ W} a_f(h)\, \mu_\infty^{\rm ST}(x)=0,\quad \int_{\widehat{T}_c/ W} a_f(h)^2\mu_\infty^{\rm ST}(x)=1.
\end{equation*}

\subsection{Spherical Maass cusp forms}

For a family of spherical Maass cusp forms in Section \ref{sec:Maass}, we
 have the following central limit theorem.

\begin{thm}\label{thm:general} 
Suppose that $G$ is a Chevalley group or a classical group. Let $\mathcal F(tD,K)$ be the family of spherical Maass cusp forms in Section \ref{sec:Maass}.
Then, under $\frac{\log t}{\log x}\to \infty$ as $x\to\infty$, we have, for $\phi$ a bounded continuous function on $\mathbb{R}$,
\[
\frac{1}{v(K)\, \mu_\infty(tD) }\sum_{f\in \mathcal{F}(tD,K)} \phi\left(\frac {\sum_{p\leq x} a_f(h_p)}{\sqrt{\pi(x)}}\right)\longrightarrow 
\frac{1}{\sqrt{2\pi}}\int_{-\infty}^\infty \phi(t) e^{-\frac {t^2}2}\, dt,\quad x\to\infty.
\]
\end{thm}
\begin{proof} 
By Proposition \ref{prop:Maass2}, one can prove our assertion as in \cite{KWY3}. 
See also Remark \ref{rem:ramanujan}. 
\end{proof}

\subsection{Holomorphic and quaternionic cusp forms}

For a family of holomorphic and quaternionic cusp forms in Section \ref{hol}, we have the following central limit theorem.

\begin{thm}\label{thm:hol} 
Suppose that $G=Sp_{2n}$ or exceptional group of type $G_2$ or $F_4$. Let $\mathcal F(\sigma_k,\tau_k,K)$ be the family of holomorphic or quaternionic cusp forms in Section \ref{hol}.
Then, under $\frac{\log k}{\log x}\to \infty$ as $x\to\infty$, for $\phi$ a bounded continuous function on $\mathbb{R}$, 
\[
\frac{1}{v(K)\, d(\sigma_k) }\sum_{f\in \mathcal{F}(\sigma_k,\tau_k,K)} \phi\left(\frac {\sum_{p\leq x} a_f(h_p)}{\sqrt{\pi(x)}}\right)\longrightarrow 
\frac{1}{\sqrt{2\pi}}\int_{-\infty}^\infty \phi(t) e^{-\frac {t^2}2}\, dt,\quad x\to\infty.
\] 
\end{thm}
\begin{proof} 
This can be proved by Proposition \ref{prop:weight2} and the same argument as in \cite{KWY3}. 
\end{proof}

\begin{remark}\label{rem:ramanujan}
In the papers \cite{KWY2, KWY3}, various assertions were proved by assuming $k_1>k_2>\cdots>k_n>n+1$ so that the generalized Ramanujan conjecture is satisfied.
However, the assumption is not necessary for the proof of the central limit theorems (Theorems \ref{thm:general} and \ref{thm:hol}).  
In fact, for each continuous function $\phi$ on $\Omega=[0,\pi]^n/S_n$ in \cite[Theorem 1.2]{KWY2}, where $S_n$ is the symmetric group on $n$ letters, there exists a continuous function $\tilde{\phi}$ on the spherical unitary dual such that the restriction of $\tilde{\phi}$ to $\Omega$ is $\phi$, and it is sufficient to replace 
$\phi$ by $\tilde{\phi}$; see \cite[Proposition 9.4 and Theorem 9.26]{ST}. 
\end{remark}

\section{Langlands $L$-functions}

Suppose $G$ is a split group of rank $n$. Let $r: \widehat{G}\longrightarrow GL_N(\Bbb C)$ be an irreducible finite dimensional representation of $\widehat{G}$.
Then $r=\lambda$, the highest weight representation for some dominant weight $\lambda$. We assume that $r$ is self-dual.

Let $\pi$ be a cuspidal automorphic representation of $G(\Bbb A)$, and let $\pi=\pi_\infty\otimes \otimes_p' \pi_p$, so that $\pi_p$ is unramified for $p\notin S$. For $p\notin S$, let $s(\pi_p)\in \widehat{T}/W$ be the Satake parameter. Then the local Langlands $L$-function associated to $\pi,r$ is $L(s,\pi_p,r)=det(I_N-r(s(\pi_p))p^{-s})^{-1}$, where $I_N$ is the $N\times N$ identity matrix.
Let $L^S(s,\pi,r)=\prod_{p\notin S} L(s,\pi_p,r)$ be the partial Langlands $L$-function. It is conjectured that $L^S(s,\pi,r)$ has an analytic continuation and satisfies a functional equation. However, for our purpose here, we do not need to assume them. 

Let
$$L^S(s,\pi,r)=\sum_{m=1}^\infty a_{\pi,r}(m) m^{-s}.
$$
Since $r$ is self-dual, $a_{\pi,r}(p)\in\Bbb R$ for all $p\notin S$.
For $p\notin S$, let $t_p=t(e^{i\theta_{1,p}},...,e^{i\theta_{n,p}})\in \widehat{G}$ be the Satake parameter attached to $\pi_p$. Then
$a_{\pi,r}(p)=Tr (r(t_p))$. Let $a_{\pi,r}=Tr (r(t))$ for $t=t(e^{i\theta_{1}},...,e^{i\theta_{n}})$.
Then $a_{\pi,r}=\chi_\lambda$, independent of $\pi$. (cf. \cite[Lemma 2.1]{ST})

Hence by Corollary \ref{chi-lambda-1},

\begin{equation} \label{Langlands}
\int_{\widehat{T}_c/W} a_{\pi,r}\, \mu_\infty^{\rm ST}(x)=0,\quad \int_{\widehat{T}_c/W} a_{\pi,r}^2\mu_\infty^{\rm ST}(x)=1.
\end{equation}


Recall the families $\mathcal{F}(tD,K)$ and $f\in\mathcal{F}(\sigma_k,\tau_k,K)$ in \S\ref{sec:Maass} and \S\ref{hol}. 
For $f\in\mathcal{F}(tD,K)$ or $f\in\mathcal{F}(\sigma_k,\tau_k,K)$, we write $\pi_f$ for the cuspidal automorphic representation generated by $f$. 
Now, a family $\mathcal{F}_j$ of automorphic representations is defined by
\[
\mathcal{F}_j\coloneqq \{ \pi_f \mid f\in \mathcal{F}(jD,K) \} \quad \text{or} \quad \{ \pi_f \mid f\in \mathcal{F}(\sigma_j,\tau_j,K) \} \qquad (j\in\mathbb{N}).
\]
Note that $\mathcal{F}_j$ is the finite multi-set consisting of cuspidal automorphic representations, that is, each cuspidal automorphic representation $\pi=\pi_\infty\otimes\pi_\mathbf{f}$ appears in $\mathcal{F}_j$ with multiplicity $\dim \pi_\mathbf{f}^K$. 

Suppose $a_{\pi,r}(p)\in I\subset \Bbb R$. 
Then the vertical Sato-Tate theorem in Propositions \ref{prop:Maass2} and \ref{prop:weight2} can be stated as 

\begin{thm}[Simultaneous vertical Sato-Tate theorem] For a bounded continuous function $\phi$ on $I^u$, and $p_1,...,p_u$ distinct primes such that $p_i\notin S$ for all $i$,
$$\frac 1{|\mathcal F_j|} \sum_{\pi\in \mathcal F_j} \phi(a_{\pi,r}(p_1),...,a_{\pi,r}(p_u))\longrightarrow \int_{I^u} \phi(x_1,...,x_u)\, \mu_{p_1}\cdots \mu_{p_u},\quad \text{as $j\to\infty$}.
$$
\end{thm}

Now Theorems 4.1 and 4.2 can be restated as 

\begin{thm}[Central limit theorem for Langlands $L$-functions] Suppose $r$ is self-dual. Then
$\{a_{\pi,r}(p), \pi\in\mathcal F_j\}$ satisfies the central limit theorem.
Namely, under $\frac{\log j}{\log x}\to \infty$ as $x\to\infty$,  
for any bounded continuous function $\phi$,

$$\frac 1{|\mathcal F_j|} \sum_{\pi\in \mathcal F_j} \phi\left(\frac {\sum_{p\leq x} a_{\pi,r}(p)}{\sqrt{\pi(x)}}\right)
\longrightarrow \frac 1{\sqrt{2\pi}}\int_{-\infty}^\infty \phi(t) e^{-\frac {t^2}2}\, dt, \quad \text{as $x\to\infty$}.
$$
\end{thm}

\begin{remark}\label{non-self-dual}
 If $r$ is not self-dual, $a_{\pi,r}(p)\notin \Bbb R$. But Proposition 3.7 and 3.9 are still valid.
Then we can prove that under $\frac{\log j}{\log x}\to \infty$ as $x\to\infty$, we have 
 \[
\frac{1}{|\mathcal F_j| }\sum_{\pi\in \mathcal{F}_j} \left(\frac {\sum_{p\leq x} a_{\pi,r}(p)}{\sqrt{\pi(x)}}\right)^a\longrightarrow \frac{1}{\pi}\iint_{\Bbb R^2} (u+iv)^a e^{-(u^2+v^2)}\, dudv=0,\quad x\to\infty.
\]
Similarly,
 \begin{eqnarray*}
&& \frac{1}{|\mathcal F_j | }\sum_{\pi\in \mathcal{F}_j} \left(\frac {\sum_{p\leq x} {\rm Re}(a_{\pi,r}(p))}{\sqrt{\pi(x)}}\right)^a=
\frac{1}{|\mathcal F_j| }\sum_{\pi\in \mathcal{F}_j} \left(\frac {\sum_{p\leq x} {\rm Im}(a_{\pi,r}(p))}{\sqrt{\pi(x)}}\right)^a \\
&&
\phantom{xxxxxxxxxxxxxxxxxxxxxxx} \longrightarrow \frac{1}{\pi}\iint_{\Bbb R^2} u^a e^{-(u^2+v^2)}\, dudv=\delta_{2|a} \frac {a!}{2^a (\frac a2)!},\quad x\to\infty.
\end{eqnarray*}

For $PGL(n)$, this is \cite[(1.5), (1.6)]{LNW}. For the proof, by the multinomial formula,
$$\left(\sum_{p\leq x} a_{\pi,r}(p)\right)^a=\sum_{u=1}^a \frac 1{u!} \text{$\sum^{(1)}_{(a_1,...,a_u)}$} \frac {a!}{a_1!\cdots a_u!} \text{$\sum_{(p_1,...,p_u)}^{(2)}$} a_{\pi,r}(p)^{a_1}\cdots a_{\pi,r}(p)^{a_u},
$$
where $\sum_{(a_1,...,a_u)}^{(1)}$ means the sum over the $u$-tuples $(a_1,...,a_u)$ of positive integers such that $a_1+\cdots+a_u=r$, and $\sum_{(p_1,...,p_u)}^{(2)}$ means the sum over the $u$-tuples $(p_1,...,p_u)$ of distinct primes not greater than $x$.
Apply the simultaneous vertical Sato-Tate theorem (Propositions \ref{prop:Maass2} and \ref{prop:weight2}). As in the proof of \cite[Theorem 3.7]{KWY3}, we can reduce to the case $(a_1,...,a_u)=(2,...,2)$.
We apply Corollary \ref{chi-lambda}, and noting that $\mu_p=(1+O(p^{-1}))\mu_\infty^{\rm ST}$, get our result.
 \end{remark}

\subsection{Symmetric power $L$-functions of $PGL_2$}

If $r=Sym^m$ and $\pi$ is a cuspidal representation of $PGL_2$, we compute directly $a_{\pi,r}$ and show that it satisfies (\ref{Langlands}).

We recall some elementary facts on symmetric functions. 
The $u$-th elementary symmetric functions $e_u$ for variables $\underline{x}=(x_1,x_2,\dots,x_n)$ are defined by
\[
\sum_{u \in \mathbb{Z}_{\ge 0}} e_u(\underline{x})\, t^u =\prod_{k=1}^n (1+x_k \, t). 
\]
The $u$-th complete symmetric functions $h_u$ for variables $\underline{x}=(x_1,x_2,\dots,x_n)$ are defined by
\[
\sum_{u \in \mathbb{Z}_{\ge 0}} h_u(\underline{x})\, t^u =\prod_{k=1}^n  (1-x_k \, t)^{-1}. 
\]
We will use the following formula. (See \cite[p.28]{M}.)
\begin{equation}\label{eq:sym}
h_u=\det(e_{1-i+j})_{1\le i,j\le u}.
\end{equation}

Recall a known fact on Hecke algebra of $PGL_2$. (See e.g. \cite[\S 7.1]{Se}.)
Let $f\in \mathcal F_k$, and let $\alpha_{f,p}$ denote the Satake parameter of $f$. Then
\[
a_f(p)=e_1(\alpha_{f,p},\alpha_{f,p}^{-1})=\alpha_{f,p}+\alpha_{f,p}^{-1}. 
\]
Furthermore, by \eqref{eq:sym} we find
\[
a_f(p^u)=h_u(\alpha_{f,p},\alpha_{f,p}^{-1})=\det\begin{pmatrix} a_f(p)&1&0& \cdots & 0 &0 \\ 1&a_f(p)& 1 & \cdots &0& 0 \\  0&1&a_f(p) & \cdots &0& 0 \\ & & & \cdots & & \\ 0& 0& 0& \cdots &a_f(p) &1 \\ 0&0&0&\cdots &1&a_f(p) \end{pmatrix}
\]
since
\[
\sum_{l=0}^\infty a_f(p^l)\, p^{-ls}=(1-\alpha_{f,p}p^{-s})^{-1}(1-\alpha_{f,p}^{-1}p^{-s})^{-1}.
\]

Set
\[
H_u(x)=\det\begin{pmatrix} x&1&0& \cdots & 0 &0 \\ 1&x& 1 & \cdots &0& 0 \\  0&1&x & \cdots &0& 0 \\ & & & \cdots & & \\ 0& 0& 0& \cdots &x &1 \\ 0&0&0&\cdots &1&x \end{pmatrix}.
\]
Then, $a_{\pi,u}=H_u$, and $H_u$ satisfies
\begin{equation}\label{eq:recurrence}
H_0(x)=1, \quad H_1(x)=x,  \quad H_{u+2}(x)=x H_{u+1}(x)-H_u(x) \;\; (u\ge 0).
\end{equation}
Let
\[
\Omega=[-2,2], \quad \mu^\mathrm{ST}_\infty(x)=\frac{1}{2\pi}\sqrt{4-x^2}.
\]
Define an inner product $\langle \; , \; \rangle$ on $\mathbb{R}[x]$ by
\[
\langle f,g \rangle =\int_\Omega f(x)\, g(x) \,  \mu^{\mathrm{ST}}_\infty(x). 
\]
Then by direct computation,
\[
\int_\Omega x^{2u}\, \mu^\mathrm{ST}_\infty(x) =\frac{(2u)!}{(u+1)! \, u!}.
\]

By the recurrence relation \eqref{eq:recurrence}, we find that $H_u(x)=U_u(\frac x2)$, where $U_u(x)$ is the Chebyshev polynomial of the second kind. 
Hence the sequence $\{H_u \}_{u\in\mathbb{Z}_{\ge 0}}$ is a system of orthogonal polynomials for $\langle \; , \; \rangle$. 
Hence we verified directly that $H_u$'s are Heckman-Opdam polynomials associated to the root system $A_1$.
In particular, 
\[
\int_\Omega H_u(x) \,  \mu^\mathrm{ST}_\infty(x)= 0 \;\; (\forall u\in\mathbb{Z}_{>0}) , \quad \int_\Omega H_u(x)^2 \,  \mu^\mathrm{ST}_\infty(x)= 1 \;\; (\forall u\in\mathbb{Z}_{\ge 0}) .
\]

Since the vertical Sato-Tate distribution is known for the family $\mathcal F_k$ of holomorphic cusp forms of weight $k$ (cf. \cite{CDF}), it follows that the central limit theorem holds for any $a_f(p^u)$, namely, we have

\begin{thm} Let $u$ be any positive integer. For any bounded continuous function $\phi: \Bbb R\rightarrow \Bbb R$,
if $\frac {\log k}{\log x}\to\infty$ as $x\to\infty$,
\begin{equation*}
\frac 1{d_k} \sum_{f\in \mathcal F_k} \phi\left(\frac {\sum_{p\leq x} a_{f}(p^u)}{\sqrt{\pi(x)}}\right)
\longrightarrow \frac 1{\sqrt{2\pi}}\int_{-\infty}^\infty \phi(t)e^{-\frac {t^2}2}\, dt, \quad x\to\infty.
\end{equation*}
\end{thm}

Since $a_f(p^u)=a_{Sym^u f}(p)$, the central limit theorem holds for $\{Sym^u f\}$.

\section{Exceptional group of type $G_2$}\label{sec:G_2}

Let $G$ denote the split algebraic group of type $G_2$ over $\Q$. In the notation of \cite{K},
let $\Delta=\{\beta_1,\cdots,\beta_6\}$ be the set of positive roots, where $\beta_1$ is the long simple root, $\beta_6$ is the short simple root; $\beta_2=\beta_1+\beta_6$, $\beta_3=2\beta_1+3\beta_6$, $\beta_4=\beta_1+2\beta_6$, $\beta_5=\beta_1+3\beta_6$. 
Here short roots are $\beta_2,\beta_4,\beta_6$, and long roots are $\beta_1,\beta_3,\beta_5$.

Let $P=MN$ be the Heisenberg parabolic subgroup associated to $\beta_1$. Then $M\simeq GL_2$ with simple root $\beta_6$. Here we choose an isomorphism so that the modulus character of $P$ is $\delta_P=det^{-3}$.
Let $Q=LU$ be the Siegel parabolic subgroup associated to $\beta_6$. Then $L\simeq GL_2$ with simple root $\beta_1$.
 Then $L^{ss}\simeq SL_2$.

Recall the setting in \S\ref{hol} for type $G_2$. Then, $\sigma_k$ is the quaternionic discrete series of $G(\Bbb R)$ with the minimal $K_\infty$-type $\tau_k$ for  $k\ge 1$. 


\subsection{Modular forms on $G_2$}

Gan-Gross-Savin \cite{GGS} developed the theory of modular forms on $G$: 
A modular form $\phi$ on $G$ of weight $k$ is a function on $G(\Bbb A)=G(\Bbb R)\times G(\Bbb A_{\bf f})$ such that
\begin{enumerate}
  \item $\phi$ is left invariant under $G(\Bbb Q)$;
  \item $\phi$ is right-invariant under some open compact subgroup of $G(\Bbb A_{\bf f})$;
  \item $\phi$ is annihilated by an ideal of finite codimension in $Z(\frak g)$;
  \item $\phi$ is of moderate growth on $G(\Bbb R)$;
  \item $\phi$ generates a representation $\pi=\pi_\infty\otimes \otimes'_p \pi_p$ of $G(\Bbb A)$ such that $\pi_\infty=\sigma_k$.
  \end{enumerate}  

Choose an open compact subgroup $K$ of $G(\A_{\bf f})$.
Let $S_k(K)$ be the space of cusp forms of weight $k$ and level $K$, that is, $S_k(K)$ consists of modular forms on $G$ which is right $K$-invariant and satisfy the cuspidal condition. 
Let $H_k(K)$ be the set of Hecke eigenforms which form a basis of $S_k(K)$. 
For $F\in H_k(K)$, let $\pi_F$ be the cuspidal automorphic representation of $G$ attached to $F$.
The family $H_k(K)$ agrees with $\mathcal{F}(\sigma_k,\tau_k,K)$ in Section \ref{hol}.
Hence, it follows from Theorem \ref{thm:STW} that
\[
|H_k(K)|= v(K)\, d(\sigma_k) +O_{\varepsilon,K}\left( k^{4+\varepsilon} \right) ,\quad k\to\infty.
\]
The factor $v(K)\, d(\sigma_k)$ is explicitly given as
\begin{equation}\label{eq:leadingg2}
v(K)\, d(\sigma_k)=\frac 1{2^9\cdot 3^4\cdot 5\cdot 7}(k-1)k(2k-1)(3k-2)(3k-1) \times \frac{[K_0:K\cap K_0]}{[K:K\cap K_0]}
\end{equation}
by using the explicit formula of  Dalal \cite{Da} for $|H_k(K_0)|$, where $K_0\coloneqq G(\widehat{\Bbb Z})$ is the maximal compact subgroup. 

\subsection{Sato-Tate measure}

Under the identification $M\simeq GL_2$, let $t=t(e^{i\theta_1},e^{i\theta_2})$ be the torus element in ${}^L G$. Then
\begin{eqnarray*}
&& \beta_1^{\vee}(t)=e^{i(\theta_1-\theta_2)},\quad \beta_6^{\vee}(t)=e^{i(-\theta_1+2\theta_2)},\quad \beta_2^{\vee}(t)=e^{i(2\theta_1-\theta_2)},\\
&& \beta_3^{\vee}(t)=e^{i\theta_1},\quad \beta_4^{\vee}(t)=e^{i(\theta_1+\theta_2)},\quad \beta_5^{\vee}(t)=e^{i\theta_2}.
\end{eqnarray*}

The Sato-Tate measure is
$$\mu_\infty^{\rm ST}=\frac 1{48\pi^2}\left| (1-e^{i(\theta_1-\theta_2)})(1-e^{i(-\theta_1+2\theta_2)})(1-e^{i(2\theta_1-\theta_2)})(1-e^{i\theta_1})(1-e^{i\theta_2})(1-e^{i(\theta_1+\theta_2)})\right|^2 \, d\theta_1 d\theta_2.
$$

Let $r: \widehat G\longrightarrow GL_7(\Bbb C)$ be the degree 7 standard representation.
Its weights are $0,\, \pm \beta_1^{\vee},\, \pm \beta_3^{\vee},\, \pm \beta_5^{\vee}$ (short roots).
Let $\pi_F=\pi_\infty\otimes \otimes_p' \pi_p$ so that $\pi_p$ is unramified for $p\notin S$. Let

$$L^S(s,\pi_F,r)=\sum_{n=1}^\infty a_{F,r}(n)n^{-s}.
$$
Then 
$$a_{F,r}(p)=1+e^{i\theta_{1,p}}+e^{-i\theta_{1,p}}+
e^{i\theta_{2,p}}+
e^{-i\theta_{2,p}}+
e^{i(\theta_{1,p}-\theta_{2,p})}+
e^{-i(\theta_{1,p}-\theta_{2,p})}.
$$
Then
$$p_r:=1+e^{i\theta_{1}}+e^{-i\theta_{1}}+
e^{i\theta_{2}}+
e^{-i\theta_{2}}+
e^{i(\theta_{1}-\theta_{2})}+
e^{-i(\theta_{1}-\theta_{2})},
$$
is the irreducible character.
Set
\[
x_1=e^{i\frac{\theta_1+\theta_2}{2}}+e^{-i\frac{\theta_1+\theta_2}{2}},\quad x_2=e^{i\frac{\theta_1-\theta_2}{2}}+e^{-i\frac{\theta_1-\theta_2}{2}}.
\]
Then 
$$
p_r= - 1+ x_1x_2+x_2^2.
$$
Also by trigonometric identities, [Note that the Jacobian is $\frac 2{\sqrt{(4-x_1^2)(4-x_2^2)}}.$]

$$\mu_\infty^{\rm ST}=\frac 1{6\pi^2}(4-x_1^2)^\frac 12(4-x_2^2)^\frac 12(x_1-x_2)^2(x_1+3x_2-x_2^3)^2\, dx_1dx_2.
$$

So we have
$$\int_{\Omega} p_r \mu_\infty^{\rm ST}=0,\quad \int_{\Omega} p_{r}^2 \mu_\infty^{\rm ST}=1,
$$
where $\Omega=[-2,2]^2/S_2$. Hence by Theorem \ref{thm:hol},
we have the central limit theorem for $G_2$:

\begin{thm} For any bounded continuous function $\phi: \Bbb R\rightarrow \Bbb R$,
if $\frac {\log k}{\log x}\to\infty$ as $x\to\infty$,
\begin{equation*}
\frac 1{|H_k(K)|} \sum_{F\in H_k(K)} \phi\left(\frac {\sum_{p\leq x} a_{F,r}(p)}{\sqrt{\pi(x)}}\right)
\longrightarrow \frac 1{\sqrt{2\pi}}\int_{-\infty}^\infty \phi(t)e^{-\frac {t^2}2}\, dt, \quad x\to\infty.
\end{equation*}
\end{thm}

\end{document}